\documentclass[11pt,a4paper,twoside]{article}
\usepackage{amsmath,amsfonts,amssymb,amsthm,bbm,latexsym,mathrsfs}
\usepackage{graphicx,color,epsfig,fancyhdr,dsfont}
\usepackage{enumerate}
\usepackage{hyperref}
\usepackage{indentfirst}
\usepackage[all]{hypcap}
\usepackage[affil-it]{authblk}
\allowdisplaybreaks
\usepackage[text={15.6cm,21.6cm},top=10mm]{geometry}
\title{Differentiable maps with isolated critical points are not necessarily open in infinite dimensional spaces}

\author[*]{Chunrong Feng}
\author[$\dag$]{Liangpan Li}
	\affil[*]{Department of Mathematical Sciences, Durham
		University, DH1 3LE, UK}
	\affil[$\dag$]{School of Mathematics, Shandong University, Jinan, Shandong, 250100, China}	

	\affil[ ]{chunrong.feng@durham.ac.uk, liliangpan@gmail.com}

\date{}

\newtheorem{thm}{Theorem}[section]

\newtheorem{defi}[thm]{Definition}
\newtheorem{rmk}[thm]{Remark}
\newtheorem{exmp}[thm]{Example}
\newtheorem{question}[thm]{Question}

\allowdisplaybreaks
\numberwithin{equation}{section}

\begin{document}

\maketitle

\begin{abstract}

Jean Saint Raymond asked whether continuously differentiable maps with isolated critical points are necessarily
open in infinite dimensional (Hilbert)
spaces. We answer this question negatively by constructing counterexamples in various settings.

\medskip

%\vskip25pt
\end{abstract}

\pagestyle{fancy}
\fancyhf{}
\fancyhead[LE,RO]{\thepage}
\fancyhead[LO]{\small{Differentiable maps with isolated critical points}}
\fancyhead[RE]{\small{C. Feng and  L. Li}}

\section{Introduction}

It is well known \cite{Bollobas,Coleman,Nirenberg} that $C^1$ (continuously differentiable) maps without critical points between Banach spaces are open.
Saint Raymond \cite{Saint} asked whether such phenomenon still occurs if the given maps  are relaxed to having isolated critical points
in  infinite dimensional (Hilbert) spaces. The purpose of the paper is to answer this question negatively by constructing
counterexamples in various real Banach spaces including all separable ones.

Back to finite dimensional spaces, Saint Raymond proved \cite{Saint} that $C^1$ vector fields with countably many critical points
are open provided that the dimension of the ambient space is higher than 1. This result was rediscovered  by the second-listed author \cite{Li}
and
is implicitly implied by Theorem 1 or 2 in \cite{Titus} by Titus and Young.
For the interest of readers, we refer to \cite{Church62BullAMS,Church62,Church73,Da,Loeb,Gamboa,Hirsch,Hu,Krantz,Ra,Ruzhansky,SRaymond,Tao11}
for some relevant works in Euclidean spaces.

Throughout the paper Banach spaces are assumed to be over the field $\mathbb{R}$ of real numbers.

Our general idea is as follows.
Let $(X,\|\cdot\|)$ be a real Banach space, and consider maps of the form
\begin{equation}\label{Form}F: x\in X\mapsto \exp(-\frac{1}{|||x|||^{s}})\cdot x\in X,\end{equation}
where $|||\cdot|||$ is another norm on $X$ that is strictly weaker than $\|\cdot\|$, i.e. there exists a constant $C>0$ such that $|||\cdot |||\leq C||\cdot||$ but both norms are not equivalent. Here $s>0$ is some real number,
and we use the convention that ${1\over 0}=\infty$ and $\exp(-\infty)=0$.
It is geometrically evident that  the unique fixed point $x=0$ of $F$ is not an interior point of
the image of the unit open ball under $F$, thus $F$ is not an open map.
After specifying the quadruple $(X,\|\cdot\|,|||\cdot|||,s)$ in several ways later on,
we will always show that $F$ is a $C^1$ map with a unique critical point at the origin.
Hence  Saint Raymond's question is answered negatively.

\section{Differentiable open maps}\label{section2.1}

\begin{defi}\label{differentiable}
A  map $F: X\rightarrow Y$ between two Banach spaces  is said to be  (Fr\'{e}chet) differentiable at a point $x\in X$
if there exists a (unique) bounded linear operator $J_F(x)$ from $X$ into $Y$,  called the derivative of $F$ at $x$, such that
\[F(x+h)-F(x)=J_F(x)h+o(||h||)\ \ \  (h\to 0).\]
Furthermore, $x$ is  said to be a regular (critical) point of $F$ if $J_F(x)$ has (has not) an inverse in $B(Y,X)$, the space of bounded linear operators from $Y$
to $X$; and
  $F$ is said to be $C^1$ if it is differentiable everywhere and
  $x \mapsto J_F(x)$ is a continuous map from $X$ to $B(X,Y)$. A map between two topological spaces is said to be open if it maps open sets to open sets.
\end{defi}

We now make an initial study of maps of the form (\ref{Form}). Take a positive constant $C$ such that $|||x|||\leq C\|x\|$
for all $x\in X$.
For any $h\in X\backslash\{0\}$,
\[\frac{\|F(h)-F(0)\|}{\|h\|}=\frac{\|F(h)\|}{\|h\|}=\exp(-\frac{1}{|||h|||^{s}})\leq \exp(-\frac{1}{C^s\|h\|^{s}}),\]
where the last term goes to zero as $h\rightarrow0$ in $(X,\|\cdot\|))$. According to Definition \ref{differentiable}, $F$ is differentiable at the origin
with derivative $J_F(0)=0$,
the zero element of $B(X):=B(X,X)$. Thus $x=0$ is a critical point of $F$.

\begin{thm}\label{thm22}
The map $F$ defined by (\ref{Form}) is not open.
\end{thm}
\begin{proof} If this was not true, then by considering $F(0)=0$, there exists a $\delta\in(0,1)$
such that for any given $y\in X$ with $\|y\|=1$, one can find an element $x_y$
in the unit open ball in $(X,\|\cdot\|)$ such that \[\frac{\delta}{2}y=F(x_y)=\exp(-\frac{1}{|||x_y|||^{s}})\cdot x_y.\]
Obviously, $x_y$ must be of the form $x_y=r_yy$ for some $r_y\in(0,1)$. Consequently,
\[\frac{\delta}{2}=\exp(-\frac{1}{r_y^s|||y|||^{s}})\cdot r_y\leq\exp(-\frac{1}{|||y|||^{s}}),\]
which implies that
\[|||y|||\geq\Big(\frac{1}{\ln\frac{2}{\delta}}\Big)^{\frac{1}{s}}.\]
Therefore,  $|||\cdot|||$ is stronger than  $\|\cdot\|$ on $X$. It is assumed that $|||\cdot|||$ is strictly weaker than  $\|\cdot\|$,
 so a contradiction is derived.
This finishes the proof of Theorem \ref{thm22}.
\end{proof}

Obviously, the map $F$ defined by (\ref{Form}) is a bijective map on $X$. The proof of Theorem \ref{thm22}
implies that the inverse $F^{-1}$ of $F$ is not continuous at $F(0)=0$ of $X$. We now can make a slightly stronger estimate
as follows. Since $|||\cdot|||$ is strictly weaker than $\|\cdot\|$, one can find elements $y_n\in X$
for large enough $n\in\mathbb{N}$, such that $\|y_n\|=1$ and $|||y_n|||=\frac{1}{n}$.
Define $z_n=\frac{y_n}{n}$, and let $x_n\in X$ be such that $F(x_n)=z_n$. Obviously, $x_n$ must be of the form $x_n=\gamma_ny_n$
for some $\gamma_n>0$.
Hence
\[\exp\Big(-\frac{n^s}{\gamma_n^s}\Big)\cdot\gamma_n=\|F(x_n)\|=\|z_n\|=\frac{1}{n},\]
or equivalently $n\gamma_n=\exp\Big(\frac{n^s}{\gamma_n^s}\Big)$, from which one can easily deduce that $\gamma_n\geq\sqrt{n}$
for sufficiently large $n$. To conclude, we see that $\|z_n\|\rightarrow0$ while $\|F^{-1}(z_n)\|\rightarrow\infty$ as $n$ goes to infinity.

\section{First example}\label{section3}

Consider  \[l^2=\Big\{x=(x_1,x_2,\cdots)\in\mathbb{R}^\mathbb{N}: \sum\limits_{k=1}^\infty x_k^2<\infty\Big\}\]  with standard norm $||x||=\Big(\sum\limits_{k=1}^\infty x_k^2\Big)^{1\over 2}$. Take $|||x|||=\Big(\sum\limits_{k=1}^\infty {{x_k^2} \over {k}}\Big)^{1\over 2}$,
another norm on $l^2$, which is strictly weaker than $\|\cdot\|$.

\begin{thm}\label{thm31}
 $F(x)=\exp(-\frac{1}{|||x|||^{2}})\cdot x$
 is a $C^1$ map on $(l^2,\|\cdot\|)$ with a unique critical point at the origin.
 \end{thm}

 \begin{proof}
In Section \ref{section2.1} we have shown that $F$ is differentiable at the origin with derivative $J_F(0)=0$.
In the following we will first show that $F$ is also differentiable  on $l^2\backslash\{0\}$.
Let $x\in l^2\backslash\{0\}$ and $h\in l^2$ be fixed such that $||h||< {{|||x|||}\over  2} $. For any $N\in {\mathbb N}$, denote
\[x^N=(x_1,x_2,\ldots,x_N,0,0,\ldots), \ \ h^N=(h_1,h_2,\ldots,h_N,0,0,\ldots).\]
  We always let $N$ be large enough so that
$
  ||h^N||<{{|||x^N|||}\over 2}
 $ and $\frac{|||x|||}{2}<|||x^N|||$.
   For such $N$,
\[f_{N}: t\in[0,1]\mapsto {-1\over {|||x^N+th^N|||^2}}={-1\over {\sum\limits_{k=1}^N {{(x_k+th_k)^2}\over k}}}\in {\mathbb R}\]
is a smooth function, and
\[||f_N||_{L^\infty}:=\sup_{t\in [0,1]} |f_N(t)|=\sup_{t\in [0,1]} {1\over {|||x^N+th^N|||^2}}\leq {16\over {|||x|||^2}}\]
since for any $t\in[0,1]$,
$$|||x^N+th^N|||\geq |||x^N|||-|||h^N|||\geq |||x^N|||-||h^N||\geq { |||x^N|||\over 2}\geq { ||x||\over 4}.$$
Moreover, direct computation shows that
\begin{align*}
f'_N(t)&=f^2_N(t)\sum_{k=1}^N {{2(x_k+th_k)}\over k}h_k,\\
f^{''}_N(t)&=2f_N^3(t)\Big(\sum_{k=1}^N {{2(x_k+th_k)}\over k}h_k\Big)^2+f^2_N(t)\sum_{k=1}^N {{2h_k^2}\over k}.
\end{align*}
Hence
\begin{align*}
\|f'_N\|_{L^{\infty}}&\leq \frac{16^2}{|||x|||^4}\cdot 3\|x\|\cdot\|h\|,\\
\|f''_N\|_{L^{\infty}}&\leq (2\cdot\frac{16^3}{|||x|||^6}\cdot 9\|x\|^2+\frac{16^2}{|||x|||^4}\cdot2)\cdot \|h\|^2.
\end{align*}
More briefly, there exists a positive constant $M$ depending only on $x$ such that
\begin{equation}\Big\|\frac{d^i}{dt^i}f_N\Big\|_{L^{\infty}}\leq M\|h\|^i\end{equation}
for $i=0,1,2$. As a consequence,
\begin{equation}\label{Bound32}\Big\|\frac{d^i}{dt^i}e^{f_N}\Big\|_{L^{\infty}}\leq \widehat{M}\|h\|^i\end{equation}
for $i=0,1,2$, where $\widehat{M}$ is another positive constant depending only $M$.
Note that
\begin{align*}
F(x+h)-F(x)&=\exp \Big(-{{1}\over {|||x+h|||^2}}\Big) \cdot (x+h)-\exp \Big(-{{1}\over{|||x|||^2}}\Big) \cdot x\\
&=\lim_{N\to \infty} \Big(e^{f_N(1)}-e^{f_N(0)}\Big)\cdot x+ e^{f_N(1)}\cdot h\\
&=\lim_{N\to \infty} \Big(\frac{d}{dt}e^{f_N(t)}\Big|_{t=0}+{1\over 2}{d^2\over {dt^2}}e^{f_N(t)}\Big|_{t=\alpha_N}\Big)\cdot x +
\Big(e^{f_N(0)}+\frac{d}{dt}e^{f_N(t)}\Big|_{t=\beta_N}\Big)\cdot h,
\end{align*}
where the existence of $\alpha_N\in(0,1)$ and $\beta_N\in(0,1)$ is ensured by  Taylor's theorem in calculus.
Considering (\ref{Bound32}), one gets
\[\big\|F(x+h)-F(x)-\lim_{N\to \infty}\frac{d}{dt}e^{f_N(t)}\Big|_{t=0}\cdot x-\lim_{N\to \infty}e^{f_N(0)}\cdot h\big\|\leq \Big(\frac{\widehat{M}\|x\|}{2}+\widehat{M}\Big)\|h\|^2,\]
where
\begin{align*}
\lim_{N\to \infty}e^{f_N(0)}&=\exp(-{1\over{|||x|||^2}}),\\
\lim_{N\to \infty}\frac{d}{dt}e^{f_N(t)}\Big|_{t=0}&=\exp(-{1\over{|||x|||^2}})\cdot{1\over{|||x|||^4}}\cdot\sum_{k=1}^\infty {{2x_kh_k}\over k}.
\end{align*}
According to the Cauchy-Schwarz inequality,  the linear  map
\[h\in l^2\mapsto\exp(-{1\over{|||x|||^2}})\cdot{1\over{|||x|||^4}}\cdot\sum_{k=1}^\infty {{2x_kh_k}\over k}\cdot x\in l^2\]
 is a bounded  operator on $(l^2,\|\cdot\|)$, where now $h$ can represent an arbitrary  element of $l^2$.
By Definition \ref{differentiable}, $F$ is differentiable at $x$ with derivative given by
\begin{equation}J_F(x)h=\exp(-{1\over{|||x|||^2}})\cdot{1\over{|||x|||^4}}\cdot\sum_{k=1}^\infty {{2x_kh_k}\over k}\cdot x+\exp(-{1\over{|||x|||^2}})\cdot h. \end{equation}
It is easy to verify by composition rule that  $x\mapsto J_F(x)$ is a continuous map from $l^2$ to $B(l^2)$, so $F$ is $C^1$.

Assuming $x$ is an arbitrary non-zero element of $l^2$, we will prove next that $J_F(x)$ is  bijective on $l^2$.
For any $y\in l^2$, we want to find an $h\in l^2$ such that
 \[\exp(-{1\over{|||x|||^2}})\cdot{1\over{|||x|||^4}}\cdot\sum_{k=1}^\infty {{2x_kh_k}\over k}\cdot x+\exp(-{1\over{|||x|||^2}})\cdot h=y,\]
or  equivalently
 \begin{equation*}
 {1\over{|||x|||^4}}\cdot\sum_{k=1}^\infty {{2x_kh_k}\over k}\cdot x+ h=\exp({1\over{|||x|||^2}})\cdot y=:\tilde y.
 \end{equation*}
Thus $h$ must be of the form $h=\tilde y+\gamma x$ for some $\gamma\in\mathbb{R}$, from which we deduce
 \[ {1\over{|||x|||^4}}\cdot\sum_{k=1}^\infty {{2x_k(\tilde y_k+\gamma x_k)}\over k}\cdot x+\gamma x=0.\]
As $x\neq 0$, we get
\begin{equation*}
\gamma=-{{{1\over{|||x|||^4}}\cdot\sum\limits_{k=1}^\infty {{2x_k\tilde y_k}\over k}}\over {1+{2\over{|||x|||^2}}}}.
\end{equation*}
Hence there exists a unique solution
\begin{equation}\label{hhh}
h=\exp({1\over{|||x|||^2}})\cdot \Bigg[y-{{{1\over{|||x|||^4}}\cdot\sum\limits_{k=1}^\infty {{2x_ky_k}\over k}}\over {1+{2\over{|||x|||^2}}}}
\cdot x\Bigg].\end{equation}
In particular, if $y=0$ then the solution must be $h=0$.
This proves that $J_F(x)$ is both surjective and injective.
Finally, it follows from Banach's isomorphism theorem that
$J_F(x)$ has an inverse in $B(l^2)$. In other words, $x$ is a regular point of $F$. This finishes the proof.
 \end{proof}

According to Theorem \ref{thm22}, the map given by Theorem \ref{thm31} is our first counterexample to Saint Raymond's question.

\begin{rmk}\label{remarkJF} We remark that (\ref{hhh}) can also be written as
\[h=\exp({1\over{|||x|||^2}})\cdot \Bigg[y-{{{1\over{|||x|||^4}}\cdot J_{|||\cdot|||^2}(x)y}\over {1+{1\over{|||x|||^4}}}\cdot J_{|||\cdot|||^2}(x)x}
\cdot x\Bigg].\]
\end{rmk}

 \begin{rmk}\label{remarkA}
Following the notations in Theorem \ref{thm31}, we remark that
 \[ x\in l^2\mapsto \exp(-\frac{1}{|||x|||^{4}})\cdot x\in l^2\]
 is another counterexample to Saint Raymond's question. The continuous differentiability of this map
 is much easier to be deduced by composition rule since as  $t\mapsto \exp(-\frac{1}{t^2})$ is a smooth function on $\mathbb{R}$, it thus
 remains only to show that
$x\mapsto |||x|||^2$
 is a $C^1$ function on $l^2$. Similarly, considering $t\mapsto \exp(-\frac{1}{t})$ is smooth on $[0,\infty)$
 and after showing $|||\cdot|||^2$ is a non-negative $C^1$ function on $l^2$,
 we immediately get that $x\mapsto \exp(-\frac{1}{|||x|||^{2}})$ is also a $C^1$ function on $l^2$.
 \end{rmk}

\newpage

 \section{Second example}

 Consider \[l^p=\Big\{x=(x_1,x_2,\cdots)\in\mathbb{R}^\mathbb{N}: \sum\limits_{k=1}^\infty |x_k|^p<\infty\Big\}\ \ \ (1\leq p<\infty)\]  with
  standard norm $||x||_p=\Big(\sum\limits_{k=1}^\infty |x_k|^p\Big)^{1\over p}$.
 Given arbitrary $1\leq p_1<p_2<\infty$, it is well known (see e.g.  \cite[p. 30]{Ding}, \cite[p. 28]{Hardy}) that $l^{p_1}$ is a proper subset of $l^{p_2}$, and  \begin{eqnarray}\label{inequality}\|x\|_{p_2}\leq\|x\|_{p_1} \end{eqnarray}
 for all $x\in l^{p_1}$. Hence $\|\cdot\|_{p_2}$, another norm on $l^{p_1}$,  is strictly weaker than $\|\cdot\|_{p_1}$.

\begin{thm}\label{prop3.1}
Assume $q\geq p$ and $q$ is an even integer. Then
$
G(x):=||x||_{q}^{q}
$
 is a $C^1$ function on $(l^p,\|\cdot\|_p)$.
 \end{thm}

 \begin{proof}
Note $q$ is an even integer, so for any $x,h\in l^p$,
 \begin{align*}
 G(x+h)-G(x)&=||x+h||_{q}^{q}-||x||_{q}^{q}\\
 &=\sum_{k=1}^\infty \Big[(x_k+h_k)^{q}-x_k^{q}\Big]\\
 &=\sum_{k=1}^\infty \sum_{j=1}^{q} {q\choose j} h_k^j x_k^{q-j}\\
 &= q \sum_{k=1}^\infty h_k x_k^{q-1}+\sum_{j=2}^{q}\sum_{k=1}^\infty {q\choose j} h_k^j x_k^{q-j}\\
 &=: I+II.
 \end{align*}
 We will estimate $I$ and $II$ in the following.
 For $I$,
 by H\"{o}lder's inequality ($\frac{1}{q}+\frac{1}{q'}=1$, $q'=\frac{q}{q-1}$ is the conjugate index of $q$) and (\ref{inequality})
 ($p_1=p\leq q=p_2$), we have
 \begin{equation*}
 \Bigg|\sum_{k=1}^\infty h_k x_k^{q-1}\Bigg|\leq ||h||_{q}\cdot ||x||_{q}^{q-1}\leq ||h||_p \cdot ||x||_p^{q-1},
 \end{equation*}
 so
 \[h\in l^p\mapsto q \sum_{k=1}^\infty h_k x_k^{q-1}\in\mathbb{R}\] is a bounded linear functional on $(l^p,\|\cdot\|_p)$.
 Similarly, for $II$,
\[
  \Bigg|\sum_{j=2}^{q}\sum_{k=1}^\infty {q\choose j} h_k^j x_k^{q-j}\Bigg|\leq
  \sum_{j=2}^{q} {q\choose j} ||h||^j_q ||x||_q^{q-j}
  \leq \sum_{j=2}^{q} {q\choose j} ||h||^j_p ||x||_p^{q-j}.
\]
  Therefore, $G$ is differentiable at $x\in l^p$ with derivative given by
  \begin{eqnarray*}
  J_G(x)h=q \sum_{k=1}^\infty h_k x_k^{q-1}\ \ \ (h\in l^2).
  \end{eqnarray*}
  Next, we will show that $x\mapsto J_G(x)$ is a continuous map from $(l^p,\|\cdot\|_p)$ to its dual space. For any $x,z,h\in l^p$,
  \begin{align*}
  J_G(x+z)h-J_G(x)h&= q \sum_{k=1}^\infty h_k \Big[(x_k+z_k)^{q-1}-x_k^{q-1}\Big]\\
  &= q\sum_{i=1}^{q-1}\sum_{k=1}^\infty h_k {q-1\choose {i}} z_k^ix_k^{q-1-i}\\
  &=q\sum_{k=1}^{\infty}h_kz_k^{q-1}+q\sum_{i=1}^{q-2} {q-1\choose {i}}\sum_{k=1}^\infty h_k z_k^ix_k^{q-1-i}\\
  &=: III+ IV.
  \end{align*}
 Similar to the method of estimating $I$ and $II$, we have for $III$ and $IV$ (applying generalized H$\ddot{\rm o}$lder's inequality to $IV$ with
 $\frac{1}{q}+\frac{i}{q}+\frac{q-1-i}{q}=1$ for each individual term),
\begin{align*}
\Bigg|\sum_{k=1}^\infty h_kz_k^{q-1}\Bigg|&\leq||h||_q\cdot ||z||_q^{q-1}\leq  ||h||_p\cdot ||z||_p^{q-1},\\
\Bigg|\sum_{k=1}^\infty h_kz_k^ix_k^{q-1-i}\Bigg|&\leq||h||_q\cdot ||z||_q^i \cdot ||x||_q^{q-1-i}\leq ||h||_p\cdot ||z||_p^i \cdot ||x||_p^{q-1-i}\ \ \ (1\leq i\leq q-2).
\end{align*}
Consequently,
\begin{equation*}
||J_G(x+z)-J_G(x)||\leq q ||z||_p^{q-1}+q \sum_{i=1}^{q-2} {q-1\choose {i}} ||z||_p^i  ||x||_p^{q-1-i}.
\end{equation*}
Hence, $x\mapsto J_G(x)$ is a continuous map from $(l^p,\|\cdot\|_p)$ to its dual space. To conclude, $G$ is a $C^1$ function on $(l^p,\|\cdot\|_p)$.
This finishes the proof.
 \end{proof}

Similar to the discussions in Sections \ref{section2.1} and \ref{section3} as well as considering Remark \ref{remarkA}
and Theorem \ref{prop3.1}, it is not difficult to show that
\[F:x\mapsto \exp(-\frac{1}{\|x\|_q^q})\cdot x \]
on $(l^p,\|\cdot\|_p)$
is a non-open $C^1$ map with a unique critical point at the origin
as long as $q>p$ (so that $\|\cdot\|_q$ is strictly weaker than $\|\cdot\|_p$ on $l^p$) is an even integer.
In particular, the solution $h\in l^p$ to $J_{F}(x)h=y$, $y\in l^p$, is uniquely given by (see also Remark \ref{remarkJF})
\begin{align*}
h
&=\exp({1\over{\|x\|_q^q}})\cdot \Bigg[y-{{{q\over{\|x\|_q^{2q}}}\cdot \sum\limits_{k=1}^{\infty}x_k^{q-1}y_k}\over {1+{q\over{\|x\|_q^{q}}}}}
\cdot x\Bigg]\\
&=\exp({1\over{\|x\|_q^q}})\cdot \Bigg[y-{{{1\over{\|x\|_q^{2q}}}\cdot J_{\|\cdot\|_q^q}(x)y}\over {1+{1\over{\|x\|_q^{2q}}}}\cdot J_{\|\cdot\|_q^q}(x)x}
\cdot x\Bigg].\end{align*}
Hence we have provided a second counterexample to Saint Raymond's question.

\begin{rmk}
We remark that the assumption of $q$ being an even integer in Theorem \ref{prop3.1}
can be replaced with $q\geq3$ being an odd integer. In the following we only establish the special case of $q=5$,
and leave the general situation
as an exercise. Note $t\mapsto|t|^5$
is a $C^1$ function on $\mathbb{R}$ with derivative $5|t|t^3$ (see also  \cite[\S 2.6]{Coleman}, \cite[p. 264]{Di} or \cite[\S 2.6]{Lieb})\footnote{Lieb and Loss \cite[\S 2.6]{Lieb} would probably write this derivative as $5|t|^3t$, while we prefer $5|t|t^3$.}. Thus for any $x,h\in l^5$,
\begin{align*}
\|x+h\|_5^5-\|x\|_5^5=&\sum_{k=1}^{\infty}\Big[|x_k+h_k|^5-|x_k|^5\Big]\\
=&\sum_{k=1}^{\infty}\int_0^15|x_k+th_k|(x_k+th_k)^3h_kdt\\
=&\sum_{k=1}^{\infty}5|x_k|x_k^3h_k+\sum_{k=1}^{\infty}\int_0^15|x_k|\Big[(x_k+th_k)^3-x_k^3\Big]h_kdt\\
& +\sum_{k=1}^{\infty}\int_0^15\Big[|x_k+th_k|-|x_k|\Big](x_k+th_k)^3h_kdt\\
=: & \sum_{k=1}^{\infty}5|x_k|x_k^3h_k+A+B.
\end{align*}
The main term, as a map of $h$, is a linear functional on  $(l^5,\|\cdot\|_5)$ with operator norm  $5\|x\|_5^4$.
For the first remainder, it follows from H\"{o}lder's inequality that
\[|A|\leq5\sum_{k=1}^{\infty}\Big[3|x_k|^3h_k^2+3x_k^2|h_k|^3+|x_k|h_k^4\Big]\leq 5\Big[3\|x\|_5^3\|h\|_5^2+
3\|x\|_5^2\|h\|_5^3+\|x\|_5\|h\|_5^4\Big].\]
For the second remainder, it follows from H\"{o}lder's and Minkowski's inequalities that
\[|B|\leq5\sum_{k=1}^{\infty}(|x_k|+|h_k|)^3h_k^2\leq 5(\|x\|_5+\|h\|_5)^3\|h\|_5^2.\]
Therefore, $G:x\mapsto\|x\|_5^5$ is  everywhere differentiable  on $(l^5,\|\cdot\|_5)$ with derivative given by the main term.
For any $x,y,h\in l^5$,
\begin{align*}
(J_G(x)-J_G(y))h&=5\sum_{k=1}^{\infty}\Big[|x_k|x_k^3-|y_k|y_k^3\Big]h_k\\
&=5\sum_{k=1}^{\infty}\Big[|x_k|x_k^3-|y_k|x_k^3\Big]h_k+5\sum_{k=1}^{\infty}\Big[|y_k|x_k^3-|y_k|y_k^3\Big]h_k\\
&=5\sum_{k=1}^{\infty}\Big[|x_k|-|y_k|\Big]x_k^3h_k+5\sum_{k=1}^{\infty}\big(x_k-y_k\big)\cdot\big(x_k^2+x_ky_k+y_k^2\big)\cdot|y_k|h_k,
\end{align*}
which, followed by generalized H\"{o}lder's inequality, implies that
\[\|J_G(x)-J_G(y)\|\leq5\|x-y\|_5\|x\|_5^3+5\|x-y\|_5(\|x\|_5^2+\|x\|_5\|y\|_5+\|y\|_5^2)\|y\|_5.\]
Consequently, $G$ is a $C^1$ function  on $(l^5,\|\cdot\|_5)$. Finally,
one can deduce from (\ref{inequality}) that $G$ is also a $C^1$ function  on $(l^p,\|\cdot\|_p)$ whenever $p\in[1,5]$.
\end{rmk}

\section{Third example}\label{3rd}

Consider  \[l^{\infty}=\Big\{x=(x_1,x_2,\cdots)\in\mathbb{R}^\mathbb{N}: \sup_k|x_k|<\infty\Big\}\]  with uniform
norm $\|x\|_{\infty}:=\sup\limits_k|x_k|$.
Obviously, \[|||x|||=\Big(\sum\limits_{k=1}^\infty {{x_k^2} \over {2^k}}\Big)^{1\over 2}\]
is another norm on $l^{\infty}$ that is strictly weaker than $\|\cdot\|_{\infty}$. For any $x,h\in l^{\infty}$,
\begin{align*}|||x+h|||^2-|||x|||^2&=\sum_{k=1}^{\infty}\frac{x_k}{2^{k-1}}h_k+\sum_{k=1}^{\infty}\frac{h_k^2}{2^k}\\
&=: V+VI.\end{align*}
Note that  \[J(x): h\mapsto\sum_{k=1}^{\infty}\frac{x_k}{2^{k-1}}h_k\]
is a bounded linear functional on $(l^{\infty},\|\cdot\|_{\infty})$ with operator norm
\[\sup\limits_{||h||_{\infty}=1} |J(x)h|=\sum_{k=1}^{\infty}\frac{|x_k|}{2^{k-1}}\leq 2\|x\|_{\infty},\] and \[|VI|\leq\sum_{k=1}^{\infty}\frac{\|h\|^2_{\infty}}{2^k}=\|h\|^2_{\infty}.\] Thus $x\mapsto|||x|||^2$ is an everywhere differentiable function on
$(l^{\infty},\|\cdot\|_{\infty})$ with derivative at $x$ given by $J(x)$.
For any $x,z\in l^{\infty}$,
\[\|J(x)-J(z)\|=\sup\limits_{||h||_{\infty}=1} |(J(x)-J(z))h|=\sup\limits_{||h||_{\infty}=1} |J(x-z)h|\leq 2\|x-z\|_{\infty},\]
which implies that $x\mapsto|||x|||^2$ is a $C^1$ function on
$(l^{\infty},\|\cdot\|_{\infty})$.
Similar to the discussions in Sections \ref{section2.1} and \ref{section3} as well as considering
 Remark \ref{remarkA}, it is not difficult to show that
\begin{equation}\label{uniform}x\mapsto \exp\Bigg(-\frac{1}{\sum\limits_{k=1}^\infty {{x_k^2} \over {2^k}}}\Bigg)\cdot x \end{equation}
on $(l^{\infty},\|\cdot\|_{\infty})$
is a non-open $C^1$ map with a unique critical point at the origin.
Hence we have provided a third counterexample to Saint Raymond's question.

\begin{rmk}\label{c0} The closed subset
\[c_0=\Big\{x=(x_1,x_2,\cdots)\in\mathbb{R}^\mathbb{N}: \lim_{k\rightarrow\infty}x_k=0\Big\}\]
of $(l^{\infty},\|\cdot\|_{\infty})$, endowed with induced norm, is a Banach space. Similar to the above example, one can show that
the map defined by
(\ref{uniform}) on $c_0$ is a non-open $C^1$ map with a unique critical point at the origin.
\end{rmk}

\section{Fourth example}
Let $(\Omega, {\cal F}, P)$ be a probability space, and denote for any $p\in[1,\infty]$,
\[L^p=L^p(\Omega, {\cal F}, P):=\{f\ \mbox{real-valued measurable on}\ \Omega: \int_\Omega |f|^p dP<\infty\}\ \ \ (1\leq p<\infty)\]
with norm $||f||_{p}:=(\int_\Omega |f|^p dP)^{1\over p}$, and
\[L^{\infty}=L^{\infty}(\Omega, {\cal F}, P):=\{f\ \mbox{essentially bounded real-valued measurable on}\ \Omega\}\]
with standard uniform norm $\|\cdot\|_{\infty}$.
According to H\"{o}lder's inequality,
\begin{equation}\label{comparison}\|f\|_{p_1}\leq\|f\|_{p_2} \ \ \ (1\leq p_1<p_2\leq\infty),\end{equation}
thus $L^{p_2}$  is a linear subspace of $L^{p_1}$.

Now let $p\in(2,\infty]$ be fixed. We  assume that  $L^{p}$ is a proper subspace of $L^2$ so that some trivial exceptions
 can be
excluded. This condition means that $\|\cdot\|_2$,
another norm on $L^p$, is strictly weaker than $\|\cdot\|_p$.

For any $f\in L^p$, define
\[G(f)=\int_\Omega f^2 dP.\]
Then for any $f, h\in L^p$,
\[G(f+h)-G(f)=2\int_\Omega fh dP+\int_\Omega h^2 dP.\]
 Let $p'$ denote the conjugate index of $p$, and note $p'<2<p$. By H\"{o}lder's inequality and (\ref{comparison}),
\[\Big|\int_\Omega fh dP\big|\leq ||f||_{p}\cdot ||h||_{p'}\leq ||f||_{p}\cdot ||h||_{p},
\]
so $$h\in L^p\mapsto 2\int_\Omega fh dP$$ is a bounded linear functional on $(L^p,\|\cdot\|_p)$ with operator norm bounded above by $2\|f\|_p$.
On the other hand, by (\ref{comparison}), we have
\[\Big|\int_\Omega h^2 dP\Big|\leq ||h||_{p} ^2.\]
Therefore, $G$ is everywhere differentiable on $(L^p,\|\cdot\|_p)$ with derivative at $f\in L^p$  given by
\[J_G(f)h=2\int_\Omega fh dP\ \ \ (h\in L^p).\]
 For any $f, g\in L^p$, \[\|J_G(f)-J_G(g)\|=\sup_{\|h\|_p=1}|(J_G(f)-J_G(g))h|=
 \sup_{\|h\|_p=1}|J_G(f-g)h|\leq2\|f-g\|_p,\]
which implies that
 $G$ is a $C^1$ function on $(L^p,\|\cdot\|_p)$.
 Similar to the discussions in Sections \ref{section2.1} and \ref{section3} as well as considering
  Remark \ref{remarkA}, it is not difficult to show
that
\begin{equation}\label{uniform22}f\mapsto \exp\Big(-\frac{1}{\int_{\Omega}f^2dP}\Big)\cdot f\end{equation}
on $(L^{p},\|\cdot\|_{p})$
is a non-open $C^1$ map with a unique critical point at the origin. So we have provided a fourth counterexample to Saint Raymond's question.

\begin{rmk}\label{cK}
Let $P$ be the Lebesgue measure
on the $\sigma$-field ${\cal F}$ of  Borel measurable subsets of $\Omega:=[0,1]$.
The family $C(\Omega)$ of all real-valued continuous functions on $\Omega$ is a closed subset
 of $L^{\infty}(\Omega, {\cal F}, P)$, hence endowed with induced uniform norm, it is a Banach space.
 Similar to the above example, one can show that the map defined by
(\ref{uniform22})
on $C(\Omega)$ is a non-open $C^1$ map with a unique critical point at the origin.
\end{rmk}

%\begin{rmk} In case  $L^{2}(\Omega, {\cal F}, P)$ is an infinite dimensional separable real Hilbert space, then it is isomorphic to $l^2$.
%So according to the example given in Theorem \ref{thm31}, $C^1$ maps with isolated critical points on $L^2$ are not necessarily open.
%\end{rmk}

\section{Weakly separable spaces}

Having studied various examples in the previous four sections, we are now able to deal with
some general Banach spaces by introducing a concept called \emph{weakly separable space}.

\begin{defi}
An infinite dimensional real Banach space $X=(X,\|\cdot\|)$
is said to be weakly separable if there exists a sequence of continuous linear functions $\{l_k\}_{k=1}^{\infty}$ on
$X$ such that
\begin{equation}\label{aleph}x=0\  \mbox{in}\ X \Longleftrightarrow l_k(x)=0\  \mbox{for all}\ k\in\mathbb{N}.
\end{equation}
\end{defi}

\begin{rmk} \label{rmk72}
(\ref{aleph}) means $x=y$ in $X$ if and only if $l_k(x)=l_k(y)$ for all $k\in\mathbb{N}$, or equivalently
\[\bigcap_{k=1}^{\infty}\big\{x\in X:l_k(x)=0\big\}=\{0\}.\]
We further remark that it is of no harm to  assume $\|l_k\|_{X^{\ast}}=1$ for all $k\in\mathbb{N}$ in (\ref{aleph}).
\end{rmk}

\begin{exmp}
Let $p\in[1,\infty]$, and define $l_k(x)=x_k$ for $x=(x_1,x_2,\cdots)\in l^p$ and $k\in\mathbb{N}$. Then it is  easy to verify that
$(l^p,\|\cdot\|_p)$ is weakly separable. In exactly the same way, one can show that $c_0$ (see Remark \ref{c0}) is also weakly separable.
\end{exmp}

\begin{exmp} Let $p\in[1,\infty]$, and let $L^p(\mathbb{R}^n)$ denote the $p$-times real-valued Lebesgue integrable  functions
on $\mathbb{R}^n$ with standard $p$-norm. To be clear, if $p=\infty$ then
$p$-times and $p$-norm are understood as essentially bounded and uniform norm, respectively.
We claim that $L^p(\mathbb{R}^n)$ is weakly separable. To prove this claim, we
first pick  all real-valued  polynomials $\{g_k\}_{k=1}^{\infty}$ with rational coefficients  on $[0,1]^n$, then define
\[l_{m_1,\ldots,m_n,k}:f\mapsto \int_{[m_1,m_1+1]\times\cdots\times[m_n,m_n+1]}f(y_1,\ldots,y_n)g_k(y_1-m_1,\ldots,y_n-m_n)dy_1\cdots dy_n\]
on $L^p(\mathbb{R}^n)$ for all $(m_1,\ldots,m_n,k)\in\mathbb{Z}^n\times\mathbb{N}$.
If $l_{m_1,\ldots,m_n,k}(f)=0$ for all $k\in\mathbb{N}$ and some fixed $(m_1,\ldots,m_n)\in\mathbb{Z}^n$, then
one can suitably apply the Stone-Weierstrass theorem \cite{Bollobas,lax} to get
\[\int_{[m_1,m_1+1]\times\cdots\times[m_n,m_n+1]}f(y_1,\ldots,y_n)h(y_1,\ldots,y_n)dy_1\cdots dy_n=0\]
for all real-valued continuous functions $h$ on $[m_1+1,m_1]\times\cdots\times[m_n,m_n+1]$,  which
 implies that $f$ vanishes almost everywhere on the same $n$-cube. This suffices to establish the claim.
\end{exmp}

\begin{exmp} In accordance with Remarks \ref{c0} and \ref{cK}, it is straightforward to show  that
any infinite dimensional closed subspace of a weakly separable Banach space is also weakly separable (see also Remark \ref{rmk72}).
\end{exmp}

\begin{thm}
Any infinite dimensional separable real Banach space is weakly separable.
\end{thm}

\begin{proof}
Let $\{x_k\}_{k=1}^{\infty}$ be a sequence of non-zero elements of an infinite dimensional separable real Banach space
$(X,\|\cdot\|)$ such that its closure is $X$. Applying the Hahn-Banach extension theorem
to $\lambda x_k\mapsto\lambda\|x_k\|$ $(\lambda\in\mathbb{R})$ on the one-dimensional linear subspace spanned by $x_k$ for each $k$,
we see that there exists a sequence of continuous linear functionals $\{l_k\in X^{\ast}\}_{k=1}^{\infty}$ with
$\|l_k\|_{X^{\ast}}=1$, such that $l_k(x_k)=\|x_k\|$ for all $k$.
Then for any $x\neq0$, by choosing a positive integer $j$
such that $\|x_{j}-x\|\leq\frac{\|x\|}{4}$, one gets
\[l_{j}(x)=l_{j}(x_{j})-l_{j}(x_{j}-x)\geq\|x_{j}\|-\|x_{j}-x\|
\geq\|x\|-2\|x_{j}-x\|\geq\|x\|-\frac{\|x\|}{2}=\frac{\|x\|}{2}>0.\]
Thus (\ref{aleph}) holds for $\{l_k\in X^{\ast}\}_{k=1}^{\infty}$. This proves that $(X,\|\cdot\|)$ is weakly separable.
\end{proof}

\begin{thm}\label{thm77}
The dual space of an infinite dimensional separable real Banach space is weakly separable.
\end{thm}

\begin{proof}
Let $\{x_k\}_{k=1}^{\infty}$ be a sequence of non-zero elements of an infinite dimensional separable real Banach space
$(X,\|\cdot\|)$ such that its closure is $X$. For each $k\in\mathbb{N}$,
\[l_k:f\in X^{\ast}\mapsto f(x_k)\in\mathbb{R}\]
is a continuous linear functional on $X^{\ast}$. Suppose $l_k(f)=0$ for some $f\in X^{\ast}$ and all $k\in\mathbb{N}$, or equivalently
$f(x_k)=0$ for all  $k\in\mathbb{N}$. Since $f$ is continuous on $X$ and $\{x_k\}_{k=1}^{\infty}$ is dense in $X$, one immediately gets
$f=0$ in $X^{\ast}$.  This proves that $X^{\ast}$ is weakly separable.
\end{proof}

\begin{rmk}
Theorem \ref{thm77} includes $l^{\infty}$ and $L^{\infty}(\mathbb{R}^n)$ as typical examples, so weakly separable Banach spaces
are not necessarily separable.
\end{rmk}

Our next result is in essence the same as the example given in Section \ref{3rd}.

\begin{thm}\label{thm79} Let $(X,\|\cdot\|)$ be a weakly separable Banach space such that (\ref{aleph}) holds for some  sequence $\{l_k\in X^{\ast}\}_{k=1}^{\infty}$ with $\|l_k\|_{X^{\ast}}=1$ for all $k\in\mathbb{N}$.  Then
\[x\mapsto \exp\Bigg(-\frac{1}{\sum\limits_{k=1}^{\infty}\frac{l_k(x)^2}{2^k}}\Bigg)\cdot x\]
on $(X,\|\cdot\|)$ is a non-open $C^1$ map with a unique critical point at the origin.
\end{thm}

\begin{proof} Since $\|l_k\|_{X^{\ast}}=1$ for all $k\in\mathbb{N}$,
\[|||x|||:=\Big(\sum_{k=1}^{\infty}\frac{l_k(x)^2}{2^k}\Big)^{\frac{1}{2}}\]
is another norm on $X$ and $|||x|||\leq\|x\|$ for all $x\in X$. For each $q\in\mathbb{N}$,
\[X_q:=\bigcap_{k=1}^q\big\{x\in X: l_k(x)=0\big\}\]
is a closed subspace of $(X,\|\cdot\|)$ with codimension $\leq q$. Since $X$
is infinite dimensional, one can pick a non-zero element $x_q$ of $X_q$ for each $q\in\mathbb{N}$. Note then
\[|||x_q|||=\Big(\sum_{k=q+1}^{\infty}\frac{l_k(x_q)^2}{2^k}\Big)^{\frac{1}{2}}\leq\frac{\|x_q\|}{\sqrt{2^q}}\ \ \ (q\in\mathbb{N}),\]
which
implies that $|||\cdot|||$ is strictly weaker than $\|\cdot\|$ on $X$.
For any $x,h\in X$,
\[|||x+h|||^2-|||x|||^2=\sum_{k=1}^{\infty}\frac{l_k(x)l_k(h)}{2^{k-1}}+|||h|||^2,\]
from which
one can easily deduce that $G:x\mapsto|||x|||^2$
on $(X,\|\cdot\|)$ is a differentiable function with derivative at $x\in X$ given by
\[J_G(x)h=\sum_{k=1}^{\infty}\frac{l_k(x)l_k(h)}{2^{k-1}}\ \ \ (h\in X).\]
For any $x,y\in X$,
\[\|J_G(x)-J_G(y)\|=\sup_{\|h\|=1}|(J_G(x)-J_G(y))h|=
 \sup_{\|h\|=1}|J_G(x-y)h|\leq\sum_{k=1}^{\infty}\frac{\|x-y\|}{2^{k-1}}=2\|x-y\|,\]
 which implies that $G$ is a $C^1$ function on $(X,\|\cdot\|)$.
 By composition rule,
 \[F:x\mapsto \exp(-\frac{1}{G(x)})\cdot x\]
on $(X,\|\cdot\|)$ is a $C^1$ map  with derivative at $x\in X$ given by
\begin{equation}\label{jacobian}
J_F(x)h=\exp(-\frac{1}{G(x)})\cdot\frac{1}{G(x)^2}\cdot J_G(x)h\cdot x+\exp(-\frac{1}{G(x)})\cdot h\ \ \ (h\in X).\end{equation}
To be clear, the value of the function $t\in[0,\infty)\mapsto \exp(-\frac{1}{t})\frac{1}{t^2}$ at $t=0$
is understood as 0.
According to the discussions in the beginning part of Section \ref{section2.1} or (\ref{jacobian}),
the origin is a critical point of $F$. Assuming next $x$ is an arbitrary non-zero element of $X$,
one can check that the solution $h\in X$, to $J_F(x)h=y$, $y\in X$, is uniquely given by
\[
h
=\exp({1\over{G(x)}})\cdot \Bigg[y-{{{1\over{G(x)^2}}\cdot J_{G}(x)y}\over {1+{1\over{G(x)^2}}}\cdot J_{G}(x)x}
\cdot x\Bigg].\]
So by Banach's isomorphism theorem, every non-zero element of $X$ is a regular point of $F$.
Recall that the non-openness of $F$ is guaranteed by Theorem \ref{thm22}. This finishes the proof of Theorem \ref{thm79}.
\end{proof}

Considering Remark \ref{rmk72} and Theorem \ref{thm79}, we see that $C^1$ maps with isolated critical points on
 weakly separable Banach spaces are not necessarily open.

\section{Further remarks}

\begin{rmk}
In sharp contrast to the infinite dimensional scenario,
differentiable maps with isolated critical points do have nice properties in finite dimensional spaces.
Apart from the open mapping property mentioned in the Introduction,
 differentiable vector fields with isolated critical points
on Euclidean spaces are local homeomorphisms provided the dimension of the ambient space is higher than two \cite{Church62,Li2}.
\end{rmk}

%We end this paper with a question concerning similar phenomenon in complex Banach spaces.

%\begin{question} So far, we have answered Saint Raymond's question
%in several particular real Banach spaces. It is interesting to see whether $C^1$ maps with isolated critical points
%on an arbitrary infinite dimensional (separable, reflexive, uniformly convex, and so on) real Banach space are not necessarily open.
%\end{question}

\begin{question}
Our general approach depends crucially on the fact that $x\mapsto|||x|||^s$ is a non-negative real-valued  function on $(X,\|\cdot\|)$
so that $x\mapsto \exp(-|||x|||^{-s})$ plays like a ``black hole" near the origin.
 Can anyone provide a counterexample to Saint Raymond's question in a class of or some particular
infinite dimensional \textbf{complex} Banach spaces?

\end{question}

\end{document}